\documentclass[11pt]{amsart}

\usepackage{amssymb,amsthm,amscd,array,stmaryrd,xypic}
\usepackage[vcentermath]{youngtab}
\usepackage[all]{xy}
\usepackage[T1]{fontenc}
\usepackage[english]{babel}
\usepackage[latin1]{inputenc}
\usepackage{comment}

\let\ssection=\section
\renewcommand{\section}{\setcounter{equation}{0}\ssection}

\newtheorem{thm}{Theorem}[section]
\newtheorem{lem}[thm]{Lemma}
\newtheorem{cor}[thm]{Corollary}
\newtheorem{prop}[thm]{Proposition}
\newtheorem{defi}[thm]{Definition}
\newtheorem{ex}[thm]{Example}
\newtheorem{rmk}[thm]{Remark}

\def\a{\alpha}
\def\b{\beta}
\def\g{\gamma}
\def\Ga{\Gamma}
\def\d{\delta}

\def\l{\lambda}

\def\m{\mu}

\newcommand{\bbR}{\mathbb{R}}
\newcommand{\bbC}{\mathbb{C}}
\newcommand{\bbN}{\mathbb{N}}
\newcommand{\bbZ}{\mathbb{Z}}

\newcommand{\bbS}{\mathbb{S}}

\newcommand{\bi}{\mathsf{i}}

\newcommand{\cA}{\mathcal{A}^\l}
\newcommand{\cAl}{\mathcal{A}^{\l,\ell}}

\newcommand{\Bv}{\mathrm{Bv}}

\newcommand{\cC}{{\mathcal{C}}}

\newcommand{\cE}{\mathcal{E}}

\newcommand{\Cinfty}{{\mathcal{C}^{\infty}}}

\newcommand{\D}{\mathcal{D}}

\newcommand{\Dlm}{\mathcal{D^{\l,\m}}}

\newcommand{\Div}{\mathrm{Div}}

\newcommand{\cF}{{\mathcal{F}}}

\newcommand{\fkg}{{\mathfrak{g}}}

\newcommand{\GL}{{\mathrm{GL}}}

\newcommand{\Gr}{{\mathrm{gr}}}

\newcommand{\mg}{\mathrm{g}}

\newcommand{\HSC}{\mathrm{CHS}}
\newcommand{\HSQ}{\mathrm{QHS}}

\newcommand{\fkh}{{\mathfrak{h}}}

\newcommand{\cJ}{\mathcal{J}}

\newcommand{\cI}{\mathcal{I}}
\newcommand{\Id}{\mathrm{Id}}

\newcommand{\cK}{\mathcal{K}}
\newcommand{\ccL}{\mathcal{L}}

\newcommand{\LD}{\mathcal{L}^{\lambda,\mu}}
\newcommand{\LDX}{\mathcal{L}_X^{\lambda,\mu}}

\newcommand{\ro}{\mathrm{o}}
\newcommand{\cO}{{\mathcal{O}}}

\newcommand{\PBW}{\mathrm{Sym}}

\newcommand{\cQ}{{\mathcal{Q}}}
\newcommand{\cQlm}{\mathcal{Q}^{\l,\m}}

\newcommand{\cN}{{\mathcal{N}}}
\newcommand{\rO}{\mathrm{O}}

\newcommand{\rR}{\mathrm{R}}

\newcommand{\Ric}{\mathrm{Ric}}
\newcommand{\cS}{{\mathcal{S}}}

\newcommand{\rS}{\mathrm{S}}

\newcommand{\SL}{\mathrm{SL}}

\newcommand{\Sl}{\mathfrak{sl}}

\newcommand{\Sp}{\mathrm{Sp}}

\newcommand{\Tr}{\mathrm{Tr}}

\newcommand{\Vect}{\mathrm{Vect}}
\newcommand{\vol}{\mathrm{vol}}

\newcommand{\half}{\frac{1}{2}}

\begin{document}

\baselineskip=15pt

\title{Higher symmetries of the Laplacian via quantization}

\author{Jean-Philippe Michel}
\address{University of Luxembourg, Campus Kirchberg, Mathematics Research Unit, 6, rue Richard Coudenhove-Kalergi, L-1359 Luxembourg City, Grand Duchy of Luxembourg}
\address{University of Li\`ege, Sart-Tilman, 12 grande traverse, B-4000 Li\`ege, Belgium}
\email{jean-philippe.michel@ulg.ac.be}
\thanks{I thank the Luxembourgian NRF for support via the AFR grant PDR-09-063.}

\keywords{Symmetry algebra, Laplacian, Quantization, Conformal geometry, Minimal nilpotent orbit, Symplectic reduction.}
\subjclass[2010]{58J10, 53A30, 70S10, 17B08, 53D20, 53D55}

\begin{abstract}
We develop a new approach, based on quantization methods, to study higher symmetries of invariant differential operators. We focus here on conformally invariant powers of the Laplacian over a conformally flat manifold and recover results of Eastwood, Leistner, Gover and \v{S}ilhan.
In particular, conformally equivariant quantization establishes a correspondence between the algebra of Hamiltonian symmetries of the null geodesic flow and the algebra of higher symmetries of the conformal Laplacian. Combined with a symplectic reduction, this leads to a quantization of the minimal nilpotent coadjoint orbit of the conformal group. 
 The star-deformation of its algebra of regular functions is isomorphic to the algebra of higher symmetries of the conformal Laplacian. Both identify with the quotient of the universal envelopping algebra by the Joseph ideal.
\end{abstract}

\maketitle

\section{Introduction}

There are a number of different notions of symmetries for a differential operator $P$ on a manifold $M$. The most basic symmetries are the vector fields $X\in\Vect(M)$ preserving the considered operator: $[P,X]=0$. More generally, symmetries can be given by differential operators $D\in\D(M)$ of arbitrary order which commute with $P$. Such symmetries obviously preserve the eigenspaces of $P$. Here, we are interested in the more general notion of \textit{higher symmetries}. They are defined as the differential operators $D_1$ satisfying $PD_1=D_2P$ for some differential operator $D_2$. Thus, they preserve the kernel of $P$ but not the other eigenspaces in general. They form a subalgebra of $\D(M)$. 

The higher symmetries given by differential operators of first order form a Lie algebra $\fkg$, which contains the vector fields preserving $P$. The determination of the space of higher symmetries of $P$, together with its algebra and $\fkg$-module structure, is of interest from at least two points of view: the integrability of the equation $P\phi=0$, with $\phi$ in the source space of $P$, and the representation theory of the $\fkg$-module $\ker P$. 

Higher symmetries have been investigated first for the Laplacian $P=\Delta$. On $\bbR^3$, Boyer, Kalnins and Miller have classified all the second order higher symmetries of the Laplacian \cite{BKM76}, which allows them to get all the possible coordinates systems separating the equation $\Delta\phi=0$. Later on, revealing the conformal nature of higher symmetries of $\Delta$, Eastwood has classified all of them on $\bbR^n$ \cite{Eas05}. In particular, he provides an explicit bijection between the higher symmetries of $\Delta$  and the traceless conformal Killing tensors. After this seminal work, higher symmetries of various operators have been investigated from the point of view of parabolic geometries \cite{ELe08,GSi09,Vla12}, using either ambient method or tractor calculus. Physics paper have also appeared on the subject \cite{BGr10,BMM12}. 

Up to constants, the Lie algebra of first order higher symmetries of $\Delta$ on $\bbR^{p,q}$ is given by $\fkg=\ro(p+1,q+1)$, which acts by conformal Killing vector fields $X\in\Vect(\bbR^{p,q})$, i.e. $L_X\mg=f_X \mg$, with $f_X\in\Cinfty(M)$ and $\mg$ the pseudo-Euclidean metric. Explicitly, for $X\in\fkg$,  we have 
\begin{equation}\label{Sym1}
\Delta (X+\l\Div X)=(X+\m\Div X)\Delta,
\end{equation}  
where $\Div$ is the divergence operator, $\l=\frac{n-2}{2n}$, $\m=\frac{n+2}{2n}$.
In \cite{Eas05}, Eastwood proves that the algebra of higher symmetries of $\Delta$ is a quotient of the universal enveloping algebra $\mathfrak{U}(\fkg)/J$. Moreover, he computes the ideal $J$ which turns to be equal to the classical Joseph ideal \cite{Jos76}.

The results in \cite{Eas05}  rely on the conformal invariance of $\Delta$ on $\bbR^{p,q}$ and hold as well on any conformally flat manifold $(M,\mg)$ \cite{GSi09}, after replacing $\Delta$ by the conformal Laplacian
$$P_1=\nabla_i\,\mg^{ij}\nabla_j-\frac{n-2}{4(n-1)}\rR,$$ 
 where $\nabla$ is the Levi-Civita connection and $\rR$ the scalar curvature. On the homogeneous model of conformal geometry, given by the product of spheres $\bbS^p\times\bbS^q$, the representation of $\fkg$ on $\ker P_1$ defined by  \eqref{Sym1} integrates to a unitary irreducible representation of the Lie group $G=\rO(p+1,q+1)$, if $p+q\geq 4$ is even  \cite{BZi91}. This is the intensively studied minimal representation of $G$, see e.g$.$ \cite{KOr03I,KOr03III}. The induced representation of $\mathfrak{U}(\fkg)$  on $\ker P_1$ has for kernel the Joseph ideal $J$, as proved in \cite{BZi91}, and coincide then with the action of higher symmetries on $\ker P_1$.

In this paper, we obtain the classification of higher symmetries of $P_1$ and their algebraic structure in a new manner, using the theory of equivariant quantization of cotangent bundles \cite{DLO99,BMa06,CSi09}. By the way, we get three new results. First, we establish that the map between traceless conformal Killing tensors and higher symmetries of the Laplacian is a restriction of the conformally equivariant quantization, which is defined on all the algebra of symmetric tensors \cite{DLO99}. Second, we identify the algebra of traceless conformal Killing tensors on $\bbS^p\times\bbS^q$ with the algebra of regular functions on $\mathcal{O}_{00}$, the  minimal nilpotent coadjoint orbit of $G$. 
Third, we provide a geometric interpretation for the algebraic structure of the space of higher symmetries of $P_1$, as the unique $\fkg$-equivariant star-deformation of the algebra of regular functions on $\mathcal{O}_{00}$, investigated in \cite{ABC94,ABr02}. We determine more generally the higher symmetries of the conformal powers of the Laplacian, denoted $P_\ell$, and thus recover the results of Gover and \v{S}ilhan \cite{GSi09}. 
The present approach can be generalized to number of cases, indeed, equivariant quantization is available for any $|1|$-graded parabolic geometry and for differential operators acting on any irreducible natural bundles \cite{CSi09}. 

Let us now detail the content of this paper. 

In Section $2$, we describe our main tools, namely the classification of conformally invariant operators on symbols \cite{ERi87,Mic11a}, the conformally equivariant quantization $\cQ^{\l,\l}$ \cite{DLO99}, parametrized by $\l\in\bbR$, and the induced star product on symbols \cite{DEO04}. 

In Section $3$, we formulate and prove our first main result. We characterize the space $\cAl$ of higher symmetries of $P_\ell$  (with $\l=\frac{n-2\ell}{2n}$) and the space $\cK^\ell$ of $s$-generalized conformal Killing tensors with $s<\ell$ \cite{NPr90}, as kernels of some conformally invariant operators. Then, we prove that conformally equivariant quantization $\cQ^{\l,\l}$ intertwines both conformally invariant operators. As a result, we get an isomorphism of $\fkg$-module $\cQ^{\l,\l}:\cK^\ell\rightarrow\cAl$, with $\fkg$ the  the Lie algebra of conformal vector fields. Note that explicit formulas are available for the conformally equivariant quantization \cite{DOv01,Lou03,Rad09,Sil09}. 
  
In section $4$, we identify the algebras $\cK^\ell$ and $\cAl$. We prove that the space $\cK$ of generalized conformal Killing tensors is the subalgebra generated by $\fkg$ of the algebra of symmetric tensors. Moreover, the spaces $\cK^\ell$ arise as quotients of $\cK$. Similarly, the spaces of symmetries $\cAl$ are obtained as quotients of the algebra $\cA$ of differential operators generated by $X+\l\Div X$ with $X\in\fkg$. We describe then all the coadjoint orbits of $G$ in the image of the moment map $\mu:T^*\bbR^{p+1,q+1}\rightarrow\fkg^*$ as symplectic reductions of the source manifold. The algebras of regular functions on the two nilpotent orbits in the image of $\m$ identify with $\cK$ and $\cK^1$. As a consequence, we get an explicit description of the symmetry algebras $\cAl$ as deformations of $\cK^\ell$. More precisely, the algebra $\cK^1$ is the algebra of regular functions on the minimal nilpotent coadjoint orbit $\mathcal{O}_{00}$ and the corresponding symmetry algebra $\mathcal{A}^{\l,1}$ of $P_1$ is isomorphic to a quotient $\mathfrak{U}(\fkg)/\cJ$. 
The ideal $\cJ$ is identified with the Joseph ideal from its defining property: this is the unique completely prime ideal in $\mathfrak{U}(\fkg)$ with associated variety $\overline{\mathcal{O}_{00}}$ \cite{Jos76}.
Finally, we build a star product on each coadjoint orbit in the image of $\mu$.
In particular, the conformally equivariant quantization induces the unique graded $\fkg$-equivariant star-product on $\mathcal{O}_{00}$, studied in  \cite{ABC94,ABr02}, and furnishes a representation of this star-product on $\ker P_1$.

\section{Conformal geometry of differential operators and of their symbols}
We introduce in this section the basic notions that we use throughout the paper. We recall two important results : the existence and uniqueness of the conformally equivariant quantization \cite{DLO99} and the classification of the conformally invariant operators on the space of symbols, as in \cite{Mic11a}.
 
\subsection{Basic definitions}
Let $M$ be a smooth manifold and $\D(M)$ be the algebra of differential operators on $\Cinfty(M)$. The algebra $\D(M)$ has a natural filtration
$$
\D_0(M)\subset\D_1(M)\subset\cdots\subset\D_k(M)\subset\cdots,
$$
where the space $\D_k(M)$, of differential operators of order $k$, is defined as the space of operators $P$ on $\cC^\infty(M)$ satisfying $[\cdots [P,f_0],\cdots],f_k]=0$ for all functions $f_0,\ldots,f_k\in\cC^\infty(M)$. The associated graded algebra $\Gr\,\D(M)$ is defined as 
$$
\cS(M)=\bigoplus_{k=0}^\infty\D_k(M)/\D_{k-1}(M)
$$
 and called the algebra of symbols. It identifies to two isomorphic algebras: the algebra of symmetric tensors $\Ga(STM)$ and the algebra of functions on $T^*M$, which are fiberwise polynomial. 
In that way, $\cS(M)$ inherits of the canonical Poisson bracket $\{\cdot,\cdot\}$ on~$T^*M$.

The canonical projections $\sigma_k:\D_k(M)\rightarrow\D_k(M)/\D_{k-1}(M)$ are called the principal symbol maps. They satisfy the two following properties
\begin{eqnarray}\label{sigma_product}
\sigma_{k+l}(AB)&=&\sigma_k(A)\sigma_l(B), \\ \label{sigma_poisson}
\sigma_{k+l-1}([A,B])&=&\{\sigma_k(A),\sigma_l(B)\},
\end{eqnarray}
for all $A\in\D_k(M)$ and $B\in\D_l(M)$.

\subsection{Actions of $\Vect(M)$ on the spaces of differential operators and symbols}

The diffeomorphisms of $M$ lift canonically to automorphisms of $\GL(M)$, the principal bundle of linear frames over $M$. Consequently, they act canonically on sections of every associated bundles to $\GL(M)$. The corresponding infinitesimal actions of the Lie algebra $\Vect(M)$ of vector fields are given by Lie derivatives. In particular, we get a $\Vect(M)$-module structure on the spaces of symbols $\cS(M)$. 

The space of $\l$-densities is defined as $\cF^\l:=\Ga(|\Lambda^n T^*M|^{\otimes\l})$, with $\l\in\bbR$. The line bundle $|\Lambda^n T^*M|^{\otimes\l}$ is the associated line bundle 
\[|\Lambda^n T^*M|^{\otimes\l} = \GL(M)\times_{\rho}\bbR,\]
where the representation $\rho$ of the group $GL(n,\bbR)$ on $\bbR$ is given by
\[\rho(A) e = \vert\det A\vert^{-\lambda} e,\quad\forall A\in
GL(n,\bbR),\;\forall e\in\bbR.\]
Via a global section $|\vol|^\l$, the $\Vect(M)$-module $\cF^\l$ identifies to the module $(\cC^\infty(M),L^\l)$, endowed with the $ \Vect(M)$-action 
\begin{equation}\label{Ll}
L_X^\l=X+\l\Div(X),
\end{equation} 
where $\Div$ is the divergence operator with respect to $|\vol|$. Note that a metric $\mg$ on $M$ defines a canonical $1$-density denoted $|\vol_\mg|$.

The $\Vect(M)$-module $\Dlm$ of differential operators from $\l$- to $\mu$-densities identifies to $(\D(M),\LD)$, with $$\LDX A=L_X^\m A-AL_X^\l,$$ 
for all $X\in\Vect(M)$ and $A\in\D(M)$. This action preserves the filtration of $\D(M)$, hence the algebra of symbols inherits of a $\Vect(M)$-action compatible with the grading. This action coincides with the $\Vect(M)$-action by Lie derivative on $$\cS^\d=\cS(M)\otimes_{\Cinfty(M)}\cF^\d,$$ for $\d=\mu-\l$.

\subsection{Conformal Lie algebra}
A conformal structure on a smooth manifold $M$ is  given by an equivalence class $[\mg]$ of pseudo-Riemannian metrics, where two metrics $h$ and $\mg$ are considered equivalent if $h=F\mg$ for some positive function $F\in\Cinfty(M)$. The signature $(p,q)$ of the metric $\mg$ is an invariant of the conformal structure. 

To each signature corresponds a canonical flat model $(\bbR^{p,q},[\eta])$, with $\eta=\mathbb{I}_p\otimes-\mathbb{I}_q$. The conformal manifold $(M,[\mg])$ is said to be conformally flat if it admits an atlas $(U_i,\phi_i)$, such that $\phi_i^*[\eta]$ coincides with the restriction of $[\mg]$ to $U_i$.

The vector fields that preserve  a conformal class $[\mg]$ are called conformal Killing vector fields. They are characterized by the equation $L_X\mg= f_X \mg$, with $L_X\mg$ the Lie derivative of $\mg$ along $X$ and $f_X\in\Cinfty(M)$. If $(M,[\mg])$ is conformally flat of dimension $p+q\geq 3$, the local conformal Killing vector fields form a sheaf of Lie algebras locally isomorphic to
$$
\fkg=\ro(p+1,q+1),
$$
which is the conformal Lie algebra of  $(\bbR^{p,q},[\eta])$.

An important example of conformally flat manifold is $\bbS^p\times\bbS^q$, viewed as a homogeneous space of $G=\rO(p+1,q+1)$. Starting from the isometric action of $G$ on the pseudo-Euclidean space $\bbR^{p+1,q+1}$, we get an action of $G$ on the space of isotropic half-lines, which identifies naturally to the manifold $\bbS^p\times\bbS^q$. Via this construction, the flat metric on $\bbR^{p+1,q+1}$ induces a conformally flat structure on $\bbS^p\times\bbS^q$, preserved by the $G$-action. 

\subsection{Conformal invariants}

The classification of differential operators, acting between natural bundles and which are invariant under the action of local conformal Killing vector fields, is the same over all conformally flat manifolds $(M,[\mg])$ of signature $(p,q)$. Using local conformal coordinates $(x^i)$, which are  such that $\mg_{ij}=F\eta_{ij}$ for a positive function~$F$, the invariant differential operators are given by the same formul\ae{}  on $(M,[\mg])$ and on $\bbR^{p,q}$. Moreover, such a classification can be deduced from the classification of morphisms of generalized Verma modules of $\fkg=\ro(p+1,q+1)$, obtained in \cite{BCo85a,BCo85b}. All results presented here can also be derived from the Weyl theory of invariants \cite{Wey97}, applied to the affine part of $\fkg$, and basic computations, see e.g.\ \cite{Mic11a}.

First, we provide the well-known classification of the conformal invariants of the $\Vect(M)$-modules of differential operators $\Dlm$ and of symbols $\cS^\d$. We write them in terms of local conformal coordinates  $(x^i,p_i)$  on $T^*M$, of the corresponding derivatives $(\partial_i,\partial_{p_i})$, and of the $1$-densities $|\vol_\mg|$ and $|\vol_\eta|$ determined by the metrics $\mg$ and $\eta$ respectively. 
\begin{prop}\label{SDinv}
On a conformally flat manifold $(M,[\mg])$, the conformal invariants of $(\cS^\d)_{\d\in\bbR}$ and $(\Dlm)_{\l,\m\in\bbR}$ are given, up to a multiplicative constant, by  
\begin{itemize}
\item $R^\ell\in\cS^\frac{2\ell}{n}$ for $\ell\in\bbN$,
\item $P_\ell\in\Dlm$ for  $\ell\in\bbN$ and $\l=\frac{n-2\ell}{2n}$, $\m=\frac{n+2\ell}{2n}$,
\end{itemize}
where $R=|\vol_\mg|^{2/n}\mg^{-1}$ and $P_\ell$ is the $\ell^{\text{th}}$ conformal power of the Laplacian. In conformal coordinates, they read locally as  $R=|\vol_\eta|^{2/n}\eta^{ij}p_ip_j$ and $P_\ell=|\vol_\eta|^{2\ell/n}(\eta^{ij}\partial_i\partial_j)^\ell$ .
\end{prop}
We refer to \cite{GPe03} and references inside for global expressions of the conformal powers of the Laplacian. 
 Since the principal symbol map is $\Vect(M)$-equivariant, conformally invariant differential operators give rise to conformally invariant symbols, but the fact that they are in correspondence is remarkable. 

Second, we present the classification of the conformally invariant differential operators on the space of symbols, as it appears in \cite{Mic11a}. It relies on the harmonic decomposition of the $\fkg$-module of symbols, namely
$$
\cS^\d=\bigoplus_{k,s\in\bbN,\, 2s\leq k}\cS^\d_{k,s}, 
$$
where $\cS^\d_{k,s}$ is the module of symbols $S$ of degree $k$ and of the form $S=R^sS_0$ with $S_0$ a traceless symbol. This means $TS_0=0$, where $T$ is the trace operator locally given by $T=\eta_{ij}\partial_{p_i}\partial_{p_j}$. The other local operators playing a role are 
$$
D=\partial_i\partial_{p_i}, \qquad G=\eta^{ij}p_i\partial_j, \qquad \Delta=\eta^{ij}\partial_i\partial_j,
$$
the divergence, gradient and Laplace operators respectively. 
\begin{thm}\cite{Mic11a}\label{ConfInv}
Let $k\geq 2s$ and $k'\geq 2s'$ be four integers, and $\d,\d'\in\bbR$. The space of conformal invariant differential operators from
$\cS^\d_{k,s}$ to $\cS^{\d'}_{k',s'}$ satisfies
\begin{itemize}
\item if $\frac{n}{2}(\d'-\d)\notin\bbZ$, it is trivial,
\item if $j=\frac{n}{2}(\d'-\d)\in\bbZ$, it is one dimensional and generated by
$$
\left\{
\begin{array}{lllcl}
R^{s'}D^dT^s,& \text{ if } \,s'-s=j,\; &k-k'=d-2j &\text{ and }& \d=1+\frac{2(k-s)-d-1}{n},\\
R^{s'}G_0^gT^s,& \text{ if }\, s'-s=j-g,\;& k-k'=s-s'-j &\text{ and }& \d=\frac{2s+1-g}{n},\\
 R^{s'}\ccL_\ell T^s,& \text{ if }\, s'-s=j-\ell,\; &k-k'=2(\ell-j) &\text{ and }& \d=\half+\frac{k-\ell}{n},
\end{array}
\right.
$$
\end{itemize}
where $G_0=\Pi_0\circ G$ with $\Pi_0$ the projection on traceless symbols and $\ccL_\ell=\Delta^\ell+a_1GD\Delta^{\ell-1}+\cdots+a_\ell G^\ell D^\ell$ for given real coefficients $a_1,\ldots,a_\ell$. 
\end{thm} 
Global expression for divergence and gradient operators can be find in \cite{DSh11}, and we refer to \cite{Wun86} for $\ccL_1$.

\subsection{Conformally equivariant quantization}
Let $\l,\m\in\bbR$ and $\d=\m-\l$. We call quantization the linear isomorphisms 
$$
\cQlm:\cS^\d\rightarrow\Dlm,
$$
which are right inverses of the principal symbol map on homogeneous symbols. This means $\sigma_k\circ\cQlm=\Id$ on $\cS^\d_k$ for all $k\in\bbN$. 

Let $\fkh$ be a subalgebra of $\Vect(M)$. Since both $\cS^\d$ and $\Dlm$ are $\Vect(M)$-modules, one can look for $\fkh$-equivariant quantization, i.e.\ maps $\cQlm$ which intertwine the $\fkh$-action. There are no such map if $\fkh=\Vect(M)$ as proved in \cite{LOv00}. On a conformally flat manifold, one can choose $\fkh=\fkg$, the conformal Lie algebra. As proved in  \cite{DLO99}, there exists a unique $\fkg$-equivariant quantization $\cQlm$ for generic values of $\d=\m-\l$.

The exceptional values of $\d$ leading to a non-unique or a non-existing conformally equivariant quantization have been classified in \cite{Sil09,Mic11a}. In particular, $\d=0$ is not an exceptional value. 
\begin{thm}\cite{Sil09,Mic11a}\label{ExistUnique}
The conformally equivariant quantization exists and is unique on $\cS^\d_{k,s}$ if and only if there is no conformally invariant differential operators from $\cS^\d_{k,s}$ to $\cS^\d$, i.e.\ for $\d\notin I_{k,s}^D\amalg \left(I_{k,s}^G\cup I_{k,s}^L\right)$ where
\begin{align} \nonumber
I^D_{k,s}&=\left\{1+\frac{2(k-s)-d-1}{n}\left|\; d\in\llbracket 1,k-2s\rrbracket \right.\right\},\\[3pt] \label{Igdl}
I^G_{k,s}&=\left\{\frac{2s+1-g}{n}\left|\; g\in\llbracket 1,s\rrbracket \right.\right\},\qquad 
I^L_{k,s}=\left\{\half+\frac{k-l}{n}\left|\; l\in\llbracket 1,s\rrbracket \right.\right\}.
\end{align}
\end{thm}
Explicit formul\ae{} are available for the conformally equivariant quantization.
As an example, we recall the one obtained by Radoux \cite{Rad09} on the space $\cS^\d_{*,0}=\bigoplus_{k\in\bbN}\cS^\d_{k,0}$ of traceless symbols. It relies on the divergence operator and the normal ordering, which are locally defined by  $D=\partial_i\partial_{p_i}$ and  $\cN:S^{i_1\cdots i_k}(x)p_{i_1}\cdots p_{i_k}\mapsto S^{i_1\cdots i_k}(x)\partial_{i_1}\cdots \partial_{i_k}$.
\begin{prop}\cite{Rad09}
Let $\d\notin\{1+\frac{2k-1-m}{n} \vert\,m=1,\ldots,k\}$, $\l\in\bbR$ and $\m=\l+\d$. On the space $\cS^\d_{k,0}$ of traceless symbols of degree $k$, the conformally equivariant quantization is given by 
\begin{equation}\label{ExplicitQ}
\cQlm=\cN\circ\left(\sum_{m=0}^{k}c^k_m D^m\right), 
\end{equation}
with $c^k_0=1$ 
and
$
c^k_{m}=\frac{k-m+n\l}{m(2k-m-1+n(1-\d))}\,c^k_{m-1},
$
for $m=1,\ldots,k$. 
\end{prop} 
The conformal equivariance of $\cQlm$ implies that it is globally well-defined over conformally flat manifolds. In the general case, including symbols with non-vanishing trace, fully explicit formul\ae{} are known only for symbols up to the degree $3$ in momenta variables $p$ \cite{DOv01,Lou03}. Moreover, \v{S}ilhan has obtained an expression for the conformally equivariant quantization in the curved case on all symbols, in terms of the tractor calculus \cite{Sil09}. 

We will need an extra statement on the conformally equivariant quantization, which is not in the literature but can be straightforwardly deduced from~\cite{DLO99}.
\begin{prop}\label{CEQ-ssmod}
Let $\l,\m\in\bbR$ with $\d=\m-\l$ and let $\cE$ be a $\fkg$-submodule of $\cS^\d$. For a  shift $\d\notin \frac{1}{n}\bbN^*$, there exists a unique $\fkg$-equivariant quantization $\cQlm:\cE\rightarrow\Dlm$.  
\end{prop}

\subsection{Conformally equivariant graded star product}\label{ParStarpro}
Let us start with standard definitions. The algebra of symbols $\cS^0$ is commutative and graded, moreover, as a subalgebra of $\Cinfty(T^*M)$, it carries a Poisson bracket denoted by $\{\cdot,\cdot\}$. A graded (or homogeneous) star product on $\cS^0$ is an associative $\bbC[[\hbar]]$-linear product $\star$ on $\cS^0\otimes\bbC[[\hbar]]$, with $\hbar$ a formal parameter. For $S_1,S_2\in\cS^0$, it is of the form $S_1\star S_2=\sum_{m\in\bbN}\left(\bi\hbar\right)^mB_m(S_1,S_2)$ and satisfies:
\begin{enumerate}
\item $B_0(S_1,S_2)=S_1S_2$,
\item $B_1(S_1,S_2)-B_1(S_2,S_1)=\{S_1,S_2\}$,
\item for all integers $k,l,m$, $B_m:\cS^0_k\otimes\cS^0_l\rightarrow\cS^0_{k+l-m}$ is a bilinear operator.
\end{enumerate} 
A frequently required extra property is the symmetry (or parity) of the star product, namely $B_m(S_1,S_2)=(-1)^mB_m(S_2,S_1)$ for all integers $m$, or equivalently $\overline{S_1\star S_2}=\overline{S_2}\star\overline{S_1}$, where $\overline{\cdot}$ is the complex conjugation. 

Let us introduce three maps: the $\bbC[[\hbar]]$-linear map $\Im:\cS^0\otimes\bbC[[\hbar]]\rightarrow \cS^0\otimes\bbC[[\hbar]]$ defined by $(\bi\hbar)^k\Id$ on $\cS^0_k$, the $\bbC[[\hbar]]$-linear extension $\cQ^\l\otimes\Id:\cS^0\otimes\bbC[[\hbar]]\rightarrow\D^{\l,\l}\otimes\bbC[[\hbar]]$  of some quantization $\cQ^\l$ and the composition $\cQ^\l_\hbar=(\cQ^\l\otimes\Id)\circ \Im$. We denote by $^*$ the adjoint operation with respect to the Hermitian product $(\phi,\psi)=\int_M\overline{\phi}\psi$, defined on complex compactly supported half-densities.
\begin{prop}
The product $\star^\l$ defined by
\begin{equation}\label{Def:star}
S_1\star^\l S_2=(\cQ^\l_\hbar)^{-1}\big(\cQ^\l_\hbar(S_1)\circ\cQ^\l_\hbar(S_2)\big),\qquad \forall S_1,S_2\in\cS^0\otimes\bbC[[\hbar]],
\end{equation}
is a graded star product on $\cS^0$. If the quantization satisfies $\cQ^{\frac{1}{2}}_\hbar(\overline{S})=\cQ^{\frac{1}{2}}_\hbar(S)^*$ for all $S\in\cS^0$, then the star product $\star^{\frac{1}{2}}$ is symmetric.
\end{prop}
\begin{proof} 
These results are classical. Using the property $\left(\sigma_k\circ\cQ^\l_\hbar\right)_{|\cS_k(M)}=(\bi\hbar)^k\Id$ of the quantization $\cQ^\l_\hbar$ and the two properties \eqref{sigma_product} and \eqref{sigma_poisson} of the principal symbol maps, one easily proves that $\star^\l$ is a graded star product. If the quantization satisfies $\cQ^{\frac{1}{2}}_\hbar(\overline{S})=\cQ^{\frac{1}{2}}_\hbar(S)^*$ for all $S\in\cS^0$, we deduce that $\overline{S_1\star^{\frac{1}{2}} S_2}=\overline{S_2}\star^{\frac{1}{2}}\overline{S_1}$  for all symbols $S_1,S_2$, thanks to the equalities
 $\cQ^{\frac{1}{2}}_\hbar(\overline{S_1\star^{\frac{1}{2}} S_2})=\left(\cQ^{\frac{1}{2}}_\hbar(S_1)\circ\cQ^{\frac{1}{2}}_\hbar(S_2)\right)^*=\cQ^{\frac{1}{2}}_{\hbar}(\bar{S_2}\star^{\frac{1}{2}}\bar{S_1})$.
\end{proof}
The action of $X\in\Vect(M)$ on $\cS^0\subset \Cinfty(T^*M)$ is given by the Hamiltonian derivation $\{\mu_X,\cdot\}$,  where $\mu_X=X^ip_i$. From a star product $\star$ on $\cS^0$, we can define a new action of $\Vect(M)$ on~$\cS^0\otimes\bbC[[\hbar]]$ via the star bracket, i.e$.$ $X\in\Vect(M)$ acts on $S\in\cS^0\otimes\bbC[[\hbar]]$ by $[\mu_X,S]_\star=\mu_X\star S- S\star\mu_X$. Let $\fkh$ be a subalgebra of $\Vect(M)$. The star product is said $\fkh$-equivariant (or strongly $\fkh$-invariant) if both induced $\fkh$-actions coincide, namely $[\mu_X,S]_\star=\bi\hbar\{\mu_X,S\}$ for all $X\in\fkh$ and all $S\in\cS^0$. As one can expect, conformally equivariant quantizations give rise to $\fkg$-equivariant star products.

\begin{prop}\cite{DLO99, DEO04}\label{ConfStar}
Let $(M,[\mg])$ be a conformally flat manifold. The star product $\star^\l$ induced by the conformally equivariant quantization $\cQ^{\l,\l}$ via equation \eqref{Def:star} is a graded $\fkg$-equivariant star product on $\cS^0$. It is symmetric if and only if $\l=\half$.
\end{prop}
It is easy to prove that all graded $\fkg$-equivariant star products on $\cS^0$ arise in that way.

\section{Classification of the higher symmetries of the conformal powers of the Laplacian}

The aim of this section is to show how conformally equivariant quantization sheds new light on the determination of higher symmetries of conformal Laplacian, initiated by Eastwood~\cite{Eas05} and pursued in \cite{ELe08} and \cite{GSi09} for conformal powers of the Laplacian, in the conformally flat case. In all this section we work over a conformally flat manifold $(M,[\mg])$ of dimension $n\geq 3$ and $P_\ell$ denotes the $\ell^{\text{th}}$ conformal power of the Laplacian, pertaining to $\Dlm$ for values of the weights henceforth fixed to $\l=\frac{n-2\ell}{2n}$, $\m=\frac{n+2\ell}{2n}$.

\subsection{Definition of higher symmetries of $P_\ell$}
Let $\l'\in\bbR$ and $(P_\ell)=\{DP_\ell \vert \, D\in\D^{\m,\l'}\}$ be the left ideal generated by $P_\ell$ in $\D^{\l,\l'}$, with either $\l'=\l$ or $\l'=\m$, depending on the context.
\begin{defi}\label{defiAl}
The space of higher symmetries of $P_\ell$ is
$$
\cAl=\{D_1\in\D^{\l,\l} \text{ such that } \exists D_2\in\D^{\m,\m}, P_\ell D_1=D_2P_\ell\}/(P_\ell).
$$
\end{defi} 
If $D_1=DP_\ell$, with $D\in\D^{\m,\l}$, the equality $P_\ell D_1= (P_\ell D)P_\ell$ holds. Hence, all elements $D_1\in(P_\ell)$ satisfy the relation $P_\ell D_1=D_2P_\ell$ and the quotient defining $\cAl$ is well-defined.
 
Clearly, $\cAl$ is a subalgebra of $\D^{\l,\l}/(P_\ell)$ and coincides with the kernel of the conformally invariant map
\begin{eqnarray}\nonumber
\HSQ:\D^{\l,\l}/(P_\ell) & \rightarrow & \D^{\l,\m}/(P_\ell)\\ \label{HSQ}
 \left[D\right] &\mapsto & [P_\ell D]
\end{eqnarray}
where $\HSQ$ stands for \textit{Quantum Higher Symmetries} and $[D]=D+(P_\ell)$.
\begin{rmk}
Resorting to conformal coordinates, higher symmetries prove to be locally the same on flat and conformally flat manifolds, but global existence can nevertheless be problematic in this more general setting. We do not address this issue and work only locally.
\end{rmk}
\begin{ex}
The higher symmetries of $P_\ell$ given by first order differential operators are the constants, acting by multiplication as zero order differential operators, and the Lie derivatives $L_X^\l$ for $X\in\fkg$. In accordance with Proposition~\ref{SDinv}, we have indeed $P_\ell L_X^\l=L_X^\m P_\ell$.
\end{ex}

\subsection{Symmetries of the null geodesic flow and generalizations}\label{Par:SymCl}

Choosing a metric $\mg\in[\mg]$, we can regard $P_1$ as acting on functions. Then, $D_1\in\D_k(M)$ is a higher symmetry if there exists $A\in\D(M)$ such that $[P_1,D_1]=AP_1$. Applying the principal symbol map and using  \eqref{sigma_poisson}, we get
$$
\{R,\sigma_k(D_1)\}\in(R),
$$ 
where $R=\sigma_2(P_1)$ (see Proposition \ref{SDinv}) and $(R)$ is the ideal generated by $R$ in $\cS^0$. Consequently, $\sigma_k(D_1)$ is constant along the Hamiltonian flow of $R$, on the level set $R=0$. Via the isomorphism $TM\cong T^*M$, provided by the metric $\mg$, this flow identifies with the null geodesic flow so that $\sigma_k(D_1)$ is a constant along the null geodesics. Thus, $\sigma_k(D_1)$ is a conformal Killing tensor. We recall their definition, using round bracket for symmetrization of indices and the partial derivatives  $(\partial_i)$ associated to local conformal coordinates $(x^i)$. Moreover, $L$ is an arbitrary tensor and $G_0=\Pi\circ G$ (see Theorem \ref{ConfInv}).
\begin{defi}
A conformal Killing $k$-tensor $K$ is defined equivalently as
\begin{itemize}
\item A symmetric traceless tensor of order $k$ s.t. $\partial_{(i_0}K_{i_1\cdots i_k)}=\mg_{(i_0i_1}L_{i_2\cdots i_k)}$,
\item A traceless symbol of degree $k$ satisfying $\{R,K\}\in (R)$,
\item  A traceless symbol of degree $k$ in the kernel of $G_0$.
\end{itemize}
\end{defi}
Easy computations lead to the equivalence between the three assertions. As for higher symmetries, we are not concerned by global existence questions and work locally.
For $k=1$, we recover the notion of conformal Killing vectors whose space identifies to the Lie algebra~$\fkg$. The conformal Killing tensors of higher orders correspond to transformations of the phase space $T^*M$ not preserving the configuration manifold $M$.
Besides, the space of conformal Killing $k$-tensors is a finite-dimensional representation of $\fkg$ which turns to be irreducible, as a consequence of Lepowsky's  generalization \cite{Lep77} of the Bernstein-Gelfand-Gelfand resolution.
We can generalize this picture to tensors (or symbols) with trace, using the conformal invariance of $G_0^{2s+1}T^s$ on $\cS^0_{k,s}$. The following definition is due to Nikitin and Prilipko \cite{NPr90}.
\begin{defi}\label{gCKT}
A $s$-generalized conformal Killing $k$-tensor $K$ is defined equivalently as
\begin{itemize}
\item A symmetric traceless tensor of order $(k-2s)$ s.t. $\partial_{(i_0}\cdots\partial_{i_{2s}}K_{i_{2s+1}\cdots i_{k})}=\mg_{(i_0i_1}L_{i_2\cdots i_k)}$,
\item  A symbol $R^sK\in\cS^0_{k,s}$ which is in the kernel of $G_0^{2s+1}T^s$.
\end{itemize}
\end{defi}
The equivalence of the two assertions relies on the equality $T^s(R^sK)=cK$, where $c$ is a constant. Again, the space of $s$-generalized conformal Killing tensors of order $k$ is an irreducible $\fkg$-module. We denote this subspace of $\cS^0_{k,s}$ by $\cK_{k,s}$ and set $\cK_{*,s}=\bigoplus_{k\geq 2s}\cK_{k,s}$,
\begin{equation}\label{cKl}
\cK=\bigoplus_{s\in\bbN}\cK_{*,s}\qquad\text{and}\qquad \cK^\ell=\cK/(R^\ell)\simeq\bigoplus_{s=0}^{\ell-1}\cK_{*,s}.
\end{equation}
\begin{rmk}
Killing $k$-tensors are symmetric tensors satisfying $\partial_{(i_0}K_{i_1\cdots i_k)}=0$, or equivalently symbols of degree $k$ satisfying $\{R,K\}=0$. One can easily check that Killing tensors are elements of $\cK$.
\end{rmk}

\subsection{From classical to quantum symmetries}

Here, we state and prove our first result : the conformally equivariant quantization $\cQ^{\l,\l}$ establishes a bijection between the two spaces of symmetries $\cAl$ and $\cK^\ell$. The existence of a $\fkg$-module isomorphism between $\cAl$ and $\cK^\ell$ was established for the first time in \cite{GSi09} via different methods. We assume that $\ell\in\bbN^*$, $\l=\frac{n-2\ell}{2n}$ and $\mu=\frac{n+2\ell}{2n}$. To state our theorem we need the following
\begin{lem}\label{lem-CEQ-quotient}
Let $(R^\ell)$ be the left ideal generated by $R^\ell$ in $\cS^0$ and $(P_\ell)$ be the ideal generated by $P_\ell$ in $\D^{\l,\l}$. The conformally equivariant quantization satisfies $\cQ^{\l,\l}((R^\ell))=(P^\ell)$ and induces then an isomorphism of $\fkg$-modules
\begin{equation}\label{CEQ-quotient}
\cQ^{\l,\l}:\cS^0/(R^\ell)\rightarrow\D^{\l,\l}/(P^\ell).
\end{equation}
Abusing notation, we call it a conformally equivariant quantization and denote it by $\cQ^{\l,\l}$.
\end{lem}
\begin{proof}
The map $S R^\ell\mapsto \cQ^{\m,\l}(S) P_\ell$ is conformally equivariant on $(R^\ell)$ and provides a right inverse to the principal symbol map on homogeneous symbols. By Proposition \ref{CEQ-ssmod}, this map coincides with $\cQ^{\l,\l}$.
\end{proof}
\begin{thm}\label{THM}
 The conformally equivariant quantization as in \eqref{CEQ-quotient} induces an isomorphism of {$\fkg$-modules} $\cQ^{\l,\l}:\cK^\ell\rightarrow \cAl$, identifying higher symmetries of $P_\ell$ with $s$-generalized conformal Killing tensors for $s<\ell$. Moreover, every $K\in\cK$ satisfies $P_\ell\cQ^{\l,\l}(K)=\cQ^{\m,\m}(K)P_\ell$.
\end{thm}
\begin{proof}
The idea of the proof is to use the conformally equivariant quantization to identify the kernel $\cAl$ of the operator $\HSQ$, see (\ref{HSQ}), with the one of an operator $\HSC$ on symbols, its name standing for {\it Classical Higher Symmetries}. 
As a consequence, we have to deal with the quotient algebras $\D^{\l,\l'}/(P_\ell)$ with $\l'=\l$ or $\l'=\m$. Clearly, the principal symbol maps descend as surjective maps $\sigma_k:\D^{\l,\l'}_k/(P_\ell)\rightarrow\bigoplus_{s=0}^{\ell-1}\cS^{\l'- \l}_{k,s}$ and, whenever it exists, the conformally equivariant quantization gives then an isomorphism of $\fkg$-modules, namely $\cQ^{\l,\l'}:\bigoplus_{s=0}^{\ell-1}\cS^{\l'- \l}_{*,s}\rightarrow\D^{\l,\l'}/(P_\ell)$, where $\cS^{\l'- \l}_{*,s}=\bigoplus_{k\in\bbN}\cS^{\l'- \l}_{k,s}$.

If $2\ell< \frac{n}{2}+1$, $\cQlm$ exists on $\bigoplus_{s=0}^{\ell-1}\cS^{\l'- \l}_{*,s}$, according to Proposition \ref{ExistUnique}. The operator $\HSC$ is then determined by the following commutative diagram of $\fkg$-modules 
\begin{equation}\label{HSQC}
\xymatrix{
\D^{\l,\l}/(P_\ell)\ar[rr]^{\HSQ} && \D^{\l,\m}/(P_\ell) \\
\bigoplus_{s=0}^{\ell-1}\cS^0_{*,s} \ar[u]^{\cQ^{\l,\l}}\ar[rr]_{\HSC}&& \bigoplus_{s=0}^{\ell-1}\cS^{\frac{2\ell}{n}}_{*,s}\ar[u]_{\cQ^{\l,\m}}
}
\end{equation}
Thus, $\HSC$ is a conformally invariant operator, and as such it should fit in the classification given in Theorem \ref{ConfInv}. On $\cS^0_{*,s}$, the operator $\HSC$ is then equal to~$R^{\ell-s-1}G_0^{2s+1}T^s$, up to a multiplicative constant. This constant cannot be zero since $\HSQ$ does not vanish on the image of $\cS^0_{*,s}$. Hence, by Definition~\ref{gCKT}, the kernel of $\HSQ$ is isomorphic to the space $\cK^\ell$ via the map $\cQ^{\l,\l}$.
\\

For arbitrary values of $\ell$, $\cQlm$ may not exist and the proof is more involved. The conformally invariant operator $\HSC$ is then defined via the following commutative diagram of $\fkg$-modules 
\begin{equation}\label{CHSsigma}
\xymatrix{
\D^{\l,\l}_k/(P_\ell)\ar[rr]^{\HSQ} && \D^{\l,\m}_{k'}/(P_\ell)\ar[d]^{\sigma_{k'}} \\
\cS^0_{k,s} \ar[u]^{\cQ^{\l,\l}}\ar[rr]_{\HSC}&& \cS^{\frac{2\ell}{n}}_{k'},
}
\end{equation}
where $k'\in\bbN$ is taken as small as possible, so that $\HSC$ does not vanish. According to Theorem \ref{ConfInv}, the operator $\HSC$ is proportional either to $R^{\ell-s-1}G_0^{2s+1}T^s$ or to $R^{\ell+s-k-\frac{n}{2}}\ccL_{k+\frac{n}{2}}T^s$, and the second case can occur only if $n/2+(k-s)\leq \ell$.

We prove that $\cQ^{\l,\l}(\ker\HSC)=\ker \HSQ$. Since $\cQ^{\l,\l}$ is a bijective linear map, we get  that $\cQ^{\l,\l}(\ker\HSC)$ contains $\ker \HSQ$. 
 We prove the converse inclusion. If $\HSC$ is proportional to $R^{\ell-s-1}G_0^{2s+1}T^s$, we obtain that $\cQ^{\l,\l}(\ker\HSC)= \ker\HSQ$  by irreducibility of the kernel of $\HSC$. 
 If $\HSC$ is proportionnal to $R^{\ell+s-k-\frac{n}{2}}\ccL_{k+\frac{n}{2}}T^s$, then $\HSC$ has for target space $\cS^{2\ell/n}_{k',s'}$ with $k'=2\ell-k-n$ and $s'=2\ell-2(k-s)-n$. According to Proposition \ref{ExistUnique}, $\cQlm$ exists on $\cS^{2\ell/n}_{k',s'}$, so that one gets the conformally invariant operator
$$
\left(\HSQ\circ\cQ^{\l,\l}-\cQlm\circ\HSC\right):\cS^{0}_{k,s}\rightarrow\D^{\l,\m}_{k''}/(P_\ell),
$$
with $k''<k'$. But the only conformally invariant operator $\cS^0_{k,s}\rightarrow\cS^{2\ell/n}_{k''}$ is zero, hence the latter operator vanishes and we get $\cQ^{\l,\l}(\ker\HSC)= \ker\HSQ$ in all cases. 

We prove that $\HSC$ is proportional to $R^{\ell-s-1}G_0^{2s+1}T^s$. Suppose it is not the case, then  $\HSC$ is proportionnal to $R^{\ell+s-k-\frac{n}{2}}\ccL_{k+\frac{n}{2}}T^s$ and its kernel is infinite dimensional. Since $\cQ^{\l,\l}(\ker\HSC)= \ker\HSQ$ for all $k,s$, the graded associated algebra to $\cAl$ satisfies $\Gr\cAl\simeq \ker\HSC$, and is then a subalgebra of $\cS^0/(R^\ell)$. If its intersection with $\cS^0_{k,s}$ is infinite dimensional, then its intersection with $\cS^0_{m}$ for $m\geq k$ is infinite dimensional also. But, as stated above, $\HSC$ is proportional to $R^{\ell-s'-1}G_0^{2s'+1}T^{s'}$ on $\cS^0_{m,s'}$, for all $s'$, if $m$ is big enough. The intersection of the kernel of $\HSC$ with $\cS^0_{m}$ is then finite dimensional. As a consequence, $\HSC$ cannot be proportionnal to $R^{\ell+s-k-\frac{n}{2}}\ccL_{k+\frac{n}{2}}T^s$

Combining the results of the two preceding paragraph, we get the desired correspondence between $\cK^\ell$ and $\cAl$ in the general case. 
\\

Now, we can define on~$\D^{\l,\l}$ a new conformally invariant operator $\HSQ_0:D\mapsto P_\ell D-\cQ^{\m,\m}\circ(\cQ^{\l,\l})^{-1}(D)P_\ell$. If $\cQlm$ exists, the commutative diagram 
$$
\xymatrix{
\D^{\l,\l}\ar[rr]^{\HSQ_0} && \D^{\l,\m} \\
\cS^0 \ar[u]^{\cQ^{\l,\l}}\ar[rr]
&& \cS^{\frac{2\ell}{n}}\ar[u]_{\cQ^{\l,\m}}
}
$$
leads to a non-vanishing conformally invariant operator on $\cS^0_{*,s}$, which, by Theorem \ref{ConfInv}, is proportional to the same operator $\HSC$ as before if $s<\ell-1$, and to the null operator otherwise. We conclude that $P_\ell\cQ^{\l,\l}(K)=\cQ^{\m,\m}(K)P_\ell$ for any $K\in\cK$ if $\cQlm$ exists. The proof in the general case is analogous.
\end{proof}
\begin{rmk}
For $\ell=1$ and $(M,\mg)$ a conformally flat Lorentzian manifold, classical and quantum symmetries for the equations of motion of a free massless particle correspond to each other: $\{R,K\}\in (R)\Longleftrightarrow [P_1,\cQ^{\l,\l}(K)]\in (P_1)$. 
\end{rmk}
\begin{rmk}
The differential operators commuting with $P_\ell$ are the higher symmetries satisfying $\cQ^{\l,\l}(K)=\cQ^{\m,\m}(K)$, $K\in\cK$. In particular, for Killing $2$-tensors $K$, one has \cite{MRS13}  
\begin{eqnarray*}
\cQ^{\a,\a}(K)& =& K^{ij}\nabla_i\nabla_j+(\nabla_iK^{ij})\nabla_j-\frac{n^2\a(1-\a)}{(n+1)(n+2)}(\nabla_i\nabla_j K^{ij})\\
&&-\frac{n^2\a(\a-1)}{(n-2)(n+1)} \Ric_{ij}K^{ij}+\frac{2n^2\a(1-\a)}{(n-2)(n-1)(n+1)(n+2)}\rR\,\mg_{ij}K^{ij},
\end{eqnarray*}
where $\a\in\bbR$, $\nabla$ is the Levi-Civita connection, $\Ric$ the Ricci tensor and $\rR$ the scalar curvature. Since $\l+\m=1$, we get $\cQ^{\l,\l}(K) =\cQ^{\m,\m}(K)$, and these operators are symmetries of $P_\ell$.
\end{rmk}
\begin{rmk}
Explicit expressions can be obtained for the higher symmetries of $P_\ell$ via the formul\ae{} for the conformally equivariant quantization  given in (\ref{ExplicitQ}) for $\ell=1$ or in \cite{Sil09} for the general case. The obtained differential operators admit analogs in the curved case, which are not necessarily higher symmetries anymore. E.g.\ all the conformal Killing $2$-tensors do not give rise to higher symmetries of the conformal Laplacian in general \cite{MRS13}.
\end{rmk}

\section{Algebras of symmetries: geometric realizations and deformations}

The aim of this section is to provide a geometric interpretation for the algebras of classical symmetries, to deduce from them the algebras of higher symmetries of $P_\ell$ and to identify the star products induced by their composition as differential operators.
In all this section we work over a conformally flat manifold $(M,[\mg])$ of dimension $n\geq 3$ and signature $(p,q)$.

\subsection{Algebras of symmetries are generated by $\fkg$}\label{ParAlgSym}

Let us give a brief reminder on universal enveloping algebra $\mathfrak{U}(\fkh)$ and symmetric algebra $\rS(\fkh)$ of an arbitrary Lie algebra $\fkh$. See e.g.\ \cite{Dix74} for more details. From the tensor algebra of $\fkh$, $\mathfrak{U}(\fkh)$ and $\rS(\fkh)$ inherit respectively a filtration $\{\mathfrak{U}_k(\fkh)\}_k$ and a grading $\rS(\fkh)=\bigoplus_k\rS_k(\fkh)$ such that $\Gr\,\mathfrak{U}(\fkh)\simeq\rS(\fkh)$. Consequently, the canonical projections $\mathfrak{U}_k(\fkh)\rightarrow\mathfrak{U}_k(\fkh)/\mathfrak{U}_{k-1}(\fkh)$ define principal symbol maps, whose right inverses are called quantizations of $\rS(\fkh)$. Two such quantizations $Q_1,Q_2:\rS(\fkh)\rightarrow\mathfrak{U(\fkh)}$ are then linear bijections and satisfy $Q_1^{-1}\circ Q_2=\Id +N$ with $N:\rS(\fkh)\rightarrow\rS(\fkh)$ a map which strictly lowers the degree.  The symmetrization map $\PBW:\rS(\fkh)\rightarrow\mathfrak{U}(\fkh)$ given by
\begin{equation}\label{Sym}
\PBW:X_{i_1}\cdots X_{i_k}\mapsto \frac{1}{k!}\sum_{\tau\in\mathfrak{S}_k}X_{\tau(i_1)}\cdots X_{\tau(i_k)}
\end{equation}
is known to define a $\fkh$-equivariant quantization of $\rS(\fkh)$ for the canonical extensions of the adjoint action of $\fkh$ to $\rS(\fkh)$ and $\mathfrak{U}(\fkh)$ \cite{Dix74}. Any other $\fkh$-equivariant quantization is then of the form $\Phi=\PBW\circ\phi$, with $\phi=\Id+N$ and $N$ a $\fkh$-equivariant map on $\rS(\fkh)$ lowering the degree.

We return to the Lie algebra $\fkg=\ro(p+1,q+1)$, acting by conformal Killing vector fields on $(M,[\mg])$. Let $\m_0:T^*M\rightarrow\fkg^*$ be the moment map. Via the defining universal properties of the algebras $\rS(\fkg)$ and $\mathfrak{U}(\fkg)$, the pullback $\mu^*_0:\fkg\rightarrow\cS^0_1$ and the Lie derivative $L^\l:\fkg\rightarrow\D^{\l,\l}_1$ (see \eqref{Ll}) extend to algebra morphisms $\mu^*_0:\rS(\fkg)\rightarrow\cS^0$ and $L^\l:\mathfrak{U}(\fkg)\rightarrow\D^{\l,\l}$. Let $\cK$ be the space of generalized conformal Killing tensors, defined in \eqref{cKl}.
\begin{prop}\label{PropIJl}
Let $\l\in\bbR$. The spaces $\cK$ and $\cA:=\cQ^{\l,\l}(\cK)$ are algebras satisfying 
$$
\cK=\mu^*_0(\rS(\fkg))\simeq\rS(\fkg)/I \qquad\text{and}\qquad \cA=L^\l(\mathfrak{U}(\fkg))\simeq \mathfrak{U}(\fkg)/J^\l,
$$ 
where $I$ is a graded ideal of $\rS(\fkg)$ and $J^\l$ is a filtered ideal of $\mathfrak{U}(\fkg))$ such that $\Gr\, J^\l\simeq I$. 

Moreover, the conformally equivariant quantization of $\cK$ lifts to a $\fkg$-equivariant quantization $\Phi^\l$ of $\rS(\fkg)$, such that the following diagram commutes
\begin{equation}\label{PBW}
\xymatrix{
\rS(\fkg)\ar[d]_{\mu^*_0}\ar[rr]^{\Phi^\l} && \mathfrak{U}(\fkg)\ar[d]^{L^\l} \\
\cK \ar[rr]_{\cQ^{\l,\l}}&& \cA
}
\end{equation}
\end{prop}
\begin{proof}
 We start with proving $\cK=\mu^*_0(\rS(\fkg))$. 
Let $\ell\in\bbN^*$. By definition, the space of higher symmetries $\cAl$ is a subalgebra of $\D^{\l,\l}/(P_\ell)$. 
Therefore, the equality $\cK^\ell=\cK/(R^\ell)$ (see \eqref{cKl}) implies that $\cK$ is a subalgebra of $\cS^0$. From $L^\l(\fkg)\subset\cAl$ we deduce $\mu^*_0(\fkg)\subset\cK$ and then $\mu^*_0(\rS(\fkg))\subset\cK$. Since the $\fkg$-module $\mu^*_0(\rS(\fkg))\cap\cS^0_{k,s}$ is clearly non-empty and $\cK_{k,s}$ is an irreducible $\fkg$-module, we get the converse inclusion. 
 
Now, we prove that $\cA=L^\l(\mathfrak{U}(\fkg))$. By semi-simplicity of $\fkg$, the finite dimensional representations of $\fkg$ are completely reducible. In particular, $\cK\cap\cS^0_k$ can be viewed as a submodule of $\rS_k(\fkg)$ for all $k\in\bbN$. This leads to the decomposition $\rS(\fkg)\simeq\cK\oplus I$ of the symmetric algebra. In other words, $\mu^*_0$ admits a $\fkg$-equivariant section. Using the embedding of $L^\l(\mathfrak{U}(\fkg))$ into $\D^{\l,\l}$, we get then the following diagram of $\fkg$-modules  
$$
\xymatrix{
\rS(\fkg)\ar[rr]^{\PBW} && \mathfrak{U}(\fkg)\ar[d]^{L^\l} \\
\cK\ar[u] \ar[rrd]_{\cQ^{\l,\l}}&& L^\l(\mathfrak{U}(\fkg)) \ar[d]\\
 && \D^{\l,\l}
}
$$
Each arrow in the latter diagram is $\fkg$-equivariant and preserves the principal symbol. Hence, uniqueness of $\cQ^{\l,\l}$ on the $\fkg$-module $\cK$ implies that the diagram is commutative, proving $\cA=L^\l(\mathfrak{U}(\fkg))$.

Since $\mu^*_0$ respects the grading, its kernel $I$ is a graded ideal, and since $L^\l$ preserves the filtration, its kernel $J^\l$ is filtered. Using the commutativity of the following diagram,
$$
\xymatrix{
\mathfrak{U}_k(\fkg)\ar[d]\ar[rr]^{L^\l} && \cA\cap\D^{\l,\l}_k\ar[d] \\
\rS_k(\fkg) \ar[rr]^{\m_0^*} && \cK_k,
}
$$
where the vertical arrows denote principal symbol maps, we get that $\Gr\, J^\l=I$.

We have proved $\rS(\fkg)\simeq\cK\oplus I$, and along the same line we get $\mathfrak{U}(\fkg)\simeq \cA\oplus J^\l$. Using again the semi-simplicity of $\fkg$, the isomorphism $J^\l_k/J^\l_{k-1}\simeq I_k$ leads to $J_k^\l\simeq I_k\oplus J^\l_{k-1}$. Thus, there exists an isomorphism of $\fkg$-modules between $I$ and $J^\l$, inverse to the symbol map. Together with the previous decomposition of $\rS(\fkg)$ and $\mathfrak{U}(\fkg)$, this ensures the existence of the quantization $\Phi^\l$ and the commutativity of the diagram \eqref{PBW}. 
\end{proof}

The proof shows that $\cK_{k,s}=\mu^*_0(\rS(\fkg))\cap\cS^0_{k,s}$. Thus, on conformally flat manifolds, the $s$-generalized conformal Killing $k$-tensors are algebraically generated from the conformal Killing vectors. This fact can also be deduced from results in \cite{CDi01}.

We recall that $\cAl$ is the algebra of higher symmetries of $P_\ell$ (see Definition \ref{defiAl}) and $\cK^\ell$ is the space of $s$-generalized conformal Killing tensors with $s<\ell$ (see \eqref{cKl}). 
\begin{cor}\label{PropPBW}
Let $\ell\in\bbN^*$ and $\l=\frac{n-2\ell}{2n}$. We have the isomorphisms of algebras 
$$
\cK^\ell=\cK/(R^\ell)\simeq\rS(\fkg)/I^\ell\qquad\text{and}\qquad \cAl=\cA/(P_\ell)\simeq\mathfrak{U}(\fkg)/J^{\l,\ell},
$$
 where the ideals are $I^\ell=I+(\mu^*_0)^{-1}(R^\ell)$ and $J^{\l,\ell}=J^\l+(L^\l)^{-1}(P_\ell)$.
\end{cor}
\begin{proof}
By definition, we have $\cK^\ell=\cK/(R^\ell)$. The equality $\cAl=\cA/(P_\ell)$ is a consequence of $\cQ^{\l,\l}(\cK^\ell)=\cAl$ and $\cQ^{\l,\l}((R^\ell))=(P_\ell)$ (see Theorem \ref{THM} and Lemma \ref{lem-CEQ-quotient} respectively). The remaining results follow from Proposition \ref{PropIJl}, $R^\ell\in\mu^*_0(\rS(\fkg))$ and again $\cQ^{\l,\l}((R^\ell))=(P_\ell)$. 
\end{proof}

\subsection{A family of coadjoint orbits of $\rO(p+1,q+1)$}

We restrict in this section to the case where $M$ is the homogeneous space $\bbS^p\times\bbS^q$ of the conformal group $G=\rO(p+1,q+1)$. This group admits a linear Hamiltonian action on $T^*\bbR^{p+1,q+1}$, hence it embeds into the symplectic linear group $\Sp(2n+4,\bbR)$, with $n=p+q$. The centralizer of $G$ in $\Sp(2n+2,\bbR)$ is isomorphic to $\SL(2,\bbR)$, and they form together a Howe dual pair, see \cite{How89}. Their moment maps are given explicitly by 
\begin{equation}\label{Def:moment}
\begin{array}{rclcrcl}
\m:T^*\bbR^{p+1,q+1}& \rightarrow&\fkg^* &\quad \mbox{and}\quad\qquad &J:T^*\bbR^{p+1,q+1}&\rightarrow&\Sl(2,\bbR)^*\\
(u,v)&\mapsto& u\wedge v& &(u,v)&\mapsto& (u\cdot v,u^2,v^2)
       \end{array} 
\end{equation}
where $u,v\in\bbR^{p+1,q+1}$ and we use the $G$-module isomorphisms $\fkg^*\simeq\fkg\simeq\Lambda^2 \bbR^{p+1,q+1}$. 

Our aim is to describe the coadjoint orbits  in the image of~$\mu$  as symplectic reductions at $0$ with respect to Lie subgroups of $\SL(2,\bbR)$. This is closely related to known results on \textit{symplectic dual pairs} \cite{BWu09}, see also \cite{ORa04}. 

The Lie subgroups of $\SL(2,\bbR)$ are generated by the flow of Hamiltonian functions in $J^*\big(\rS(\Sl(2,\bbR))\big)$, i.e$.$ polynomial functions in $x^2=\eta_{AB}x^Ax^B$, $xp=x^Ap_A$ and $p^2=\eta^{AB}p_Ap_B$, where $(x^A,p_A)$ are Cartesian coordinates on $T^*\bbR^{p+1,q+1}$. Important such functions are given by the Casimir elements of $\fkg$ and $\Sl(2,\bbR)$ in $\Cinfty(T^*\bbR^{p+1,q+1})$. They are equal to 
$$
C=(xp)^2-x^2p^2
$$
 and $C/4$ respectively, if we define the Killing form by the map $(X,Y)\mapsto\half\Tr(\rho(X)\rho(Y))$ with $\rho$ the standard representation.

We denote by $\left\langle f_1,\ldots,f_k\right\rangle$ the Lie group generated by the flow of Hamiltonian functions $f_1,\ldots,f_k\in\Cinfty(T^*\bbR^{p+1,q+1})$ and by $T^*\bbR^{p+1,q+1}//\left\langle f_1,\ldots,f_k\right\rangle$ the corresponding symplectic quotient at $0$. If the linear span of those functions is closed under the Poisson bracket, the above symplectic quotient is  then the quotient of the common zero locus of $f_1,\ldots,f_k$ by their Hamiltonian flows. By the Marsden-Weinstein theorem, this quotient space is a symplectic manifold if $0$ is a regular value of the involved Hamiltonian functions. E.g$.$, we have 
\begin{equation}\label{TRTM}
T^*\left(\bbR^{p+1,q+1}\setminus \{0\}\right)//\left\langle xp,x^2\right\rangle\simeq T^*(\bbS^p\times\bbS^q).
\end{equation}
Note that $T^*(\bbS^p\times\bbS^q)$ splits into three stable submanifolds under the Hamiltonian $G$-action, according to the sign of the norm of covectors, with straightforward notation: $T^*(\bbS^p\times\bbS^q)=T^*_+(\bbS^p\times\bbS^q)\sqcup T^*_0(\bbS^p\times\bbS^q)\sqcup T^*_{-}(\bbS^p\times\bbS^q)$. 

\begin{thm}\label{thmCoad}
Let $p,q\geq 1$, $n\geq 3$ and $P(\a,\b)$ be the space of planes in $\bbR^{p+1,q+1}$ of signature $(\a,\b)$. The coadjoint orbits of $G$ in the image of $\mu$ are classified as follows:
\begin{enumerate}
\item the one parameter family of semi-simple orbits $\mathcal{O}_{a^+}$ and $\mathcal{O}_{a^-}$ for $a\in\bbR^*_+$ such that
$$
\xymatrix{
T^*\bbR^{p+1,q+1}//\left\langle xp,C-a\right\rangle
\ar[rr]^{\quad\qquad\bbZ_4}&&\mathcal{O}_{a^+}\sqcup\mathcal{O}_{a^-}\ar[rr]^{\simeq\quad}&& P(2,0)\sqcup P(0,2),
}
$$
\item the one parameter family of semi-simple orbits $\mathcal{O}_{a}$ for $a\in\bbR^*_-$ such that
$$
\xymatrix{
T^*\bbR^{p+1,q+1}//\left\langle xp,C-a\right\rangle
\ar[rr]^{\qquad\qquad\bbZ_4}&&\mathcal{O}_{a}\ar[rr]^{\simeq\quad}&& P(1,1),
}
$$
\item the two nilpotent orbits $\mathcal{O}_{0^+}$ and $\mathcal{O}_{0^-}$ such that
$$
\xymatrix{
T^*_+(\bbS^p\times\bbS^q)\sqcup T^*_-(\bbS^p\times\bbS^q)
\ar[rr]^{\quad\qquad\bbZ_2}&&\mathcal{O}_{0^+}\sqcup\mathcal{O}_{0^-}\ar[rr]^{\bbR^*\quad}&& P(1,0)\sqcup P(0,1),
}
$$
\item The minimal nilpotent orbit $\mathcal{O}_{00}$ such that
$$
\xymatrix{
(T^*(\bbS^p\times\bbS^q)\setminus \bbS^p\times\bbS^q)//\left\langle R\right\rangle\ar[rr]^{\quad\qquad\qquad\bbZ_2}&&\mathcal{O}_{00}\ar[rr]^{\bbR^*}&&P(0,0),
}
$$
\item The null orbit $\{0\}$.
\end{enumerate}
All the arrows denote $G$-equivariant coverings, whose fibers are indicated as superscript. The first ones are symplectomorphisms.
\end{thm}
\begin{proof}
Through the $G$-module isomorphisms $\Lambda^2 \bbR^{p+1,q+1}\simeq \fkg\simeq\fkg^*$, coadjoint orbits are identified to $G$-orbits in the space of bivectors, endowed with the natural $G$-action. The moment map $\mu$ defined by \eqref{Def:moment} takes its values in the space of simple bivectors $\Bv=\{u\wedge v\vert\, u,v\in\bbR^{p+1,q+1}\}$. Our key tool is the $G$-equivariant projection of $\Bv$ on the Grassmannian $\mathrm{Gr}(2,n+2)$ of planes in $\bbR^{p+1,q+1}$. This is encompassed in the following sequence of $G$-spaces:
\begin{equation}\label{BvPlan}
\xymatrix{
T^*\bbR^{p+1,q+1}\ar[rr]^{\quad\SL(2,\bbR)}&&\Bv\ar[rr]^{\bbR^*\quad\qquad}&&\mathrm{Gr}(2,n+2)\cup\{0\},
}
\end{equation}
where the superscripts denote the fibers of the coverings over $\Bv\setminus\{0\}$ and $\mathrm{Gr}(2,n+2)$. The moment map $\m$ preserves the Poisson structure, hence a $G$-stable subset of $T^*\bbR^{p+1,q+1}$ projects onto coadjoint orbits of $G$, which themselves project onto $G$-orbits of $\mathrm{Gr}(2,n+2)\cup\{0\}$. Thanks to the Witt Theorem, the latter are known to be $\{0\}$ and the $6$ spaces $P(\a,\b)$ of planes of given signature $(\a,\b)$ for the induced metric. 
From \eqref{TRTM}, we easily get that $T^*_{\pm}(\bbR^{p+1,q+1}_*)//\left\langle xp,x^2\right\rangle\simeq T^*_{\pm}(\bbS^p\times\bbS^q)$ and also that  $$\left(T^*(\bbR^{p+1,q+1}_*)\setminus(\bbR^{p+1,q+1}_*)\right)//\left\langle xp,x^2,p^2\right\rangle\;\simeq\; (T^*(\bbS^p\times\bbS^q)\setminus (\bbS^p\times\bbS^q))//\left\langle R\right\rangle,$$
where $\bbR^{p+1,q+1}_*:=\bbR^{p+1,q+1}\setminus\{0\}$.
Thus, in the four non-trivial cases, we deal with symplectic reductions of $G$-stable subset of $T^*\bbR^{p+1,q+1}$. We easily check that the common zero locus of the Hamiltonian functions defining the reduction have the announced images in $\mathrm{Gr}(2,n+2)$. 
Moreover, the one parameter groups generated by the Hamiltonian flows of $xp$, $x^2$, $p^2$ and $C$ are respectively given by $(u,v)\mapsto(e^t u,e^{-t}v)$, $(u,v)\mapsto(u,v+tu)$, $(u,v)\mapsto(u+tv,v)$ and $(u,v)\mapsto(u+(tu^2)v,v-(tv^2)u)$, for $t\in\bbR$. They act only in the fibers of $\mu$, hence the map $\mu$ descends to the symplectic quotients. A direct computation proves that the fibers of $\m$ on the reduced spaces are of cardinal $4$ or $2$. They admit a transitive action of the discrete groups $\bbZ_4$ and $\bbZ_2$ respectively, the action of their generators being $(u,v)\mapsto (-v,u)$ and $(u,v)\mapsto (-u,-v)$.  We end with the four sequences $(1)$, $(2)$, $(3)$, and $(4)$. There, a unique coadjoint orbit lies over each orbit in $\mathrm{Gr}(2,n+2)$, since the action of the group $G$ is transitive in the fibers of each arrow. 
For a proof of the minimality of $\mathcal{O}_{00}$ we refer to \cite{Wol78}. 
\end{proof}
The points $(3)$ and $(4)$ in the latter theorem combine, according to Cordani \cite{Cor86}, to provide a conformal regularization by $T^*M$ of the cone $\mathcal{O}_{0^+}\cup\mathcal{O}_{00}\cup\mathcal{O}_{0^-}$, with singularity in~$\mathcal{O}_{00}$.
\begin{rmk}
The used symplectic reductions of $T^*\bbR^{p+1,q+1}$ correspond to symplectic reduction with respect to the moment map $J$ of $\SL(2,\bbR)$ at, respectively, the points $(0,\pm\sqrt{a},\mp\sqrt{a})$ for a>0, $(0,\sqrt{|a|},\sqrt{|a|})$ for a<0, $(0,0,\pm 1)$ and~$(0,0,0)$. Hence, we obtain a bijection between the coadjoints orbits of $\SL(2,\bbR)$ and the ones in the image of $\mu$. Similar results are obtained in \cite{BWu09} for general dual pairs, under the \textnormal{symplectic Howe condition}. 
\end{rmk}
Now, we determine the algebra of regular functions on each coadjoint orbit of $G$ in the image of $\mu$. 
We have $\fkg\simeq\Lambda^2\bbR^{p+1,q+1}$, that we represent by the Young diagram ${\Yboxdim{5pt} \yng(1,1)}$. Accordingly, elementary representation theory of the orthogonal Lie algebra leads to
\begin{equation}
\fkg\odot\fkg={\mbox{\tiny $\yng(2,2)\oplus\yng(1,1,1,1)$}} \quad \text{and} \quad {\mbox{\tiny $\yng(2,2)={\yng(2,2)}_{\,0}\oplus{\yng(2)}_{\,0}$}}\oplus\bbR.
\end{equation}  
In the second decomposition, the index $0$ denotes the trace-free part, and the three components correspond to $\cK_{2,0}$, $\cK_{2,1}$ and the one-dimensional space generated by the Casimir element in~$\rS_2(\fkg)$, still denoted by $C$. The extra term in the decomposition of $\fkg\odot\fkg$ is generated by exterior products in $\Lambda\bbR^{p+1,q+1}$ of elements of $\fkg\simeq\Lambda^2\bbR^{p+1,q+1}$.
\begin{lem}
The kernel of the pullback $\mu^*:\rS(\fkg)\rightarrow\Cinfty(T^*\bbR^{p+1,q+1})$ by the moment map of $\fkg$ is the ideal generated by ${ \Yboxdim{5pt} \yng(1,1,1,1)}$. 
\end{lem}
\begin{proof}
Since elements of $\fkg$ are skew-symmetric $2$-tensors $V^{AB}$ on $\bbR^{p+1,q+1}$, the map $\mu^*$ is explicitly given by $V^{AB\cdots CD}\mapsto x_A\cdots x_CV^{AB\cdots CD}p_B\cdots p_D$, and vanishes then on tensors $V^{AB\cdots CD}$ which are skew-symmetric in any $3$ indices. Hence, the module ${ \Yboxdim{5pt} \yng(1,1,1,1)}$ is in the kernel of~$\mu^*$. In consequence, $\mu^*(\rS_k(\fkg))$ is contained in the module $\rS_k(\fkg)/\left(\,{ \Yboxdim{5pt} \yng(1,1,1,1)}\,\right)$, described by the Young diagram with $2$ lines and $k$ columns. But none of the irreducible components of such a Young diagram is in the kernel of $\mu^*$, as all the traces, $x^2, xp, p^2$, can occur in $\Cinfty(T^*\bbR^{p+1,q+1})$. In conclusion, the algebra  $\mu^*(\rS(\fkg))$ is isomorphic to $\rS(\fkg)/\left(\,{ \Yboxdim{5pt} \yng(1,1,1,1)}\,\right)$.
\end{proof}

\begin{prop}\label{PolyO}
Let $a\in\bbR$, $C$ the Casimir element of $\fkg$ and $\cI_a=\left([C-a]\bbR\oplus{\Yboxdim{5pt}\yng(1,1,1,1)}\,\right)$ an ideal of $\rS(\fkg)$. The algebras of regular functions  on $\mathcal{O}_{a^{(\pm)}}$ are given by $\rS(\fkg)/\cI_a$.
\end{prop} 
\begin{proof}
Following the proof of Theorem \ref{thmCoad}, we get that the moment map $\m$ descends to $T^*(\bbR^{p+1,q+1}\setminus\{0\})//\left\langle xp,C\right\rangle$ and provides thus a $\bbZ_4$-covering of the two orbits $\mathcal{O}_{0^\pm}$. Hence, for all  $a \in\bbR$, the coadjoint orbit $\mathcal{O}_{a^{(\pm)}}$ admits a $\bbZ_4$-covering by a symplectic reduction of $T^*\bbR^{p+1,q+1}$. 
Moreover, the generator of $\bbZ_4$ acts by $(u,v)\mapsto(-v,u)$, so that it leaves invariant the functions in $\m^*\fkg$, which are linear combinations of $x_Ap_B-x_Bp_A$, for $A,B=0,\ldots,p+q+1$. Therefore, on each coadjoint orbit $\mathcal{O}_{a^{(\pm)}}$, the algebra of regular functions is isomorphic to the reduction of $\mu^*(\rS(\fkg))$ by $\left\langle xp,C-a\right\rangle$. The reduction with respect to $xp$ modifies only the fibers of $\mu$ and the Casimir element $C$ Poisson commutes with all elements in $\mu^*(\rS(\fkg))$, so that reduction with respect to $\left\langle xp,C-a\right\rangle$ amounts to modding out by $(C-a)$. 
\end{proof}

\subsection{The algebras of classical and quantum symmetries}

We return now to a general conformally flat manifold $(M,[\mg])$, and use notation of Section \ref{ParAlgSym}. In particular, $\cK$ denotes the algebra of generalized conformal Killing tensors, generated by $\m_0^*(\fkg)$ in $\Cinfty(T^*M)$, and $\cK^1=\cK/(R)$ is the algebra of traceless conformal Killing tensors over $M$. 
\begin{thm}\label{ThmK}
The algebra $\cK$ is isomorphic to the algebras of regular functions on the orbits $\mathcal{O}_{0^{\pm}}$, given by $\rS(\fkg)/\cI_0$ with $\cI_0=\left(C\cdot\bbR\oplus{\Yboxdim{5pt}\yng(1,1,1,1)}\,\right)$ and $C$ the Casimir element of $\fkg$.
Moreover, the algebra $\cK^1$ is isomorphic to the algebra of regular functions on $\mathcal{O}_{00}$, given by $\rS(\fkg)/\cI_{00}$ with $\cI_{00}=\left({\Yboxdim{5pt}{\yng(2)}_{\,0}}\right)+\cI_0$. 
\end{thm}
\begin{proof}
The algebras $\cK$ and $\cK^1$ are of local nature and thus we can assume that $M=\bbS^p\times\bbS^q$. According to Proposition \ref{PropIJl}, the algebra $\cK$ is generated by $\m_0^*\fkg$, where $\m_0:T^*_{\pm}M\rightarrow\fkg^*$ is the moment map of $\fkg$. By Theorem \ref{thmCoad}, $\m_0$ is a $\bbZ_2$-covering of the coadjoint orbits $\mathcal{O}_{0^\pm}$ and the action of $\bbZ_2$ leaves invariant the functions in $\m_0^*\fkg$. Hence, $\cK$ identifies with the algebra of regular functions on the orbits $\mathcal{O}_{0^\pm}$ and the isomorphism $\cK\simeq\rS(\fkg)/\cI_0$ follows then from Proposition \ref{PolyO}. Similarly, the coadjoint orbit $\mathcal{O}_{00}$ admits a $\bbZ_2$-covering by the symplectic quotient $(T^*M\setminus M)//\left\langle R \right\rangle$. Thus, the algebra of regular functions on $\mathcal{O}_{00}$ arises as a reduction of $\cK$. Since $\{R,\cK\}\subset(R)$, this reduced algebra is $\cK/(R)\simeq\cK^1$. As $R\in\cK_{2,1}$ is the pullback of an element in ${\Yboxdim{5pt}{\yng(2)}_{\,0}}$, we finally obtain $\cK^1\simeq\rS(\fkg)/\cI_{00}$ and $\cI_{00}=\left({\Yboxdim{5pt}{\yng(2)}_{\,0}}\right)+\cI_0$. . 
\end{proof}
We can now recover the description of the algebras of higher symmetries $\cAl$ of the $\ell^{\text{th}}$ conformal powers of the Laplacian $P_\ell$, obtained originally in \cite{GSi09}. In addition, we determine the algebra $\cA=\cQ^{\l,\l}(\cK)$, generated by the space  of Lie derivatives $L^\l(\fkg)$ in $\D^{\l,\l}$. Recall that we define the Killing form by $\half\Tr(XY)$, for every $X,Y\in\fkg$. The corresponding Casimir operator in $\mathfrak{U}(\fkg)$ is given by $\cC=\PBW(C)$, the symmetrization of $C\in\rS(\fkg)$ (see \eqref{Sym}). 
\begin{thm}
For every $\l\in\bbR$, the algebras $\cA=\cQ^{\l,\l}(\cK)$ are isomorphic to $\mathfrak{U}(\fkg)/J^\l$ with $J^\l=\big(\PBW\left({ \Yboxdim{5pt} \yng(1,1,1,1)}\right)\oplus[\cC-\rho(\l)]\bbR\big)$, where $\rho(\l)=n^2\l(1-\l)$ is the eigenvalue of the Casimir operator $\cC$ on $\l$-densities. 

For $\l=\frac{n-2\ell}{n}$, the algebra of higher symmetries $\cAl$ is isomorphic to $\mathfrak{U}(\fkg)/J^{\l,\ell}$, where $J^{\l,\ell}$ is generated by $J^\l$ and the Young diagram \hspace{-0.4cm}
{\renewcommand{\arraystretch}{0.4} 
{\setlength{\tabcolsep}{0.07cm} 
\begin{tabular}{|c|c|c|}
 \hline
 & ... &\\
\hline
\end{tabular}
}}\hspace{-0.1cm}$_0$ 
of length $2\ell$. In particular, $J^{\l,1}$ is the Joseph ideal.
\end{thm} 
\begin{proof}
According to Proposition \ref{PropIJl} and Theorem \ref{ThmK}, we have $\cA\simeq\mathfrak{U}(\fkg)/J^\l$ and the graded ideal associated to $J^\l$ is $\cI_0$. We deduce that $J^\l$ is also generated by quadratic elements and we get then $J^\l_2=\Phi^\l (I_2)$, with $I_2=\mathcal{I}_0\cap\rS_2(\fkg)$ and $\Phi^\l=\PBW\circ\phi^\l$ (see Proposition  \ref{PropIJl}). The map $\phi^\l$ being $\fkg$-equivariant, the space $\phi^\l(I_2)$ is a $\fkg$-submodule of $\mathfrak{U}_2(\fkg)\simeq\bbR\oplus\fkg\oplus\rS_2(\fkg)$. Hence, $J_2^\l$ is generated by $\PBW\left({ \Yboxdim{5pt} \yng(1,1,1,1)}\right)$ and the Casimir operator $\cC$ of $\mathfrak{U}(\fkg)$, modified by some real number $\rho(\l)$. Since $\cC-\rho(\l)$ projects onto $0$ via $L^\l:\mathfrak{U}(\fkg)\rightarrow\D^{\l,\l}$, the real number $\rho(\l)$ is necessarily the eigenvalue of $L^\l(\cC)$ on $\l$-densities. The latter has been computed in \cite{DLO99} for the opposite Killing form.

From Corollary \ref{PropPBW}, we deduce that $\cAl$ is isomorphic to $\mathfrak{U}(\fkg)/J^{\l,\ell}$, where $J^{\l,\ell}$ is generated by $J^\l$ and the Young diagram \hspace{-0.4cm}
{\renewcommand{\arraystretch}{0.4} 
{\setlength{\tabcolsep}{0.07cm} 
\begin{tabular}{|c|c|c|}
 \hline
 & ... &\\
\hline
\end{tabular}
}}\hspace{-0.1cm}$_0$ 
of length $2\ell$. Thanks to Theorem~\ref{THM}, we have the isomorphism of algebras $\Gr\,(\mathcal{A}^{\l,1})\simeq\cK^1$. As $\cK^1\simeq\rS(\fkg)/\cI_{00}$ and the ideal $\cI_{00}$ is prime, we deduce that $J^{\l,1}$ is completely prime. Besides, their common characteristic variety is the closure of the minimal nilpotent coadjoint orbit of~$G$. These two properties characterize the Joseph ideal \cite{Jos76}. 
\end {proof}
The identification of the Joseph ideal in the context of the higher symmetries of the Laplacian was already obtained in different manners \cite{ESS07, Som08}, but not from its original definition like here. The determination of the ideals $J^{\l,\ell}$ has been already performed in the context of higher symmetries of $P_\ell$ in \cite{Eas05,ELe08,GSi09}, but in different terms. Let us make clear the link between the two approaches.  We denote by $\left\langle \cdot,\cdot\right\rangle$ the chosen Killing form and $C$ the associated Casimir element in $\rS(\fkg)$.  In the previous works, the projections of $X\odot Y\in\fkg\odot\fkg$ on each irreducible component are used. Following $\fkg\odot\fkg={ \Yboxdim{5pt}\yng(2,2)}_{\,0}\oplus{ \Yboxdim{5pt}\yng(2)}_{\,0}\oplus\bbR\oplus{ \Yboxdim{5pt}\yng(1,1,1,1)}$, we have $X\odot Y=X\boxtimes Y+X\bullet Y+\frac{\left\langle X,Y\right\rangle}{2\dim\fkg}C+ X\wedge Y$. Then, the ideal $J^\l$ is clearly generated by $\PBW\big(\frac{\left\langle X,Y\right\rangle}{2\dim\fkg}(C-\rho(\l))+ X\wedge Y\big)$ for $X,Y\in\fkg$ or equivalently by 
$$
\PBW\big(X\odot Y-X\boxtimes Y-X\bullet Y+\frac{\rho(\l)}{2\dim\fkg}\left\langle X,Y\right\rangle\big),
$$
which is the obtained expression in \cite{Eas05,ELe08,GSi09}, modulo the extra generator associated to $R^\ell$.

\subsection{Quantization of a family of coadjoint orbits of $G$}

Let $H$ be a Lie group with Lie algebra $\fkh$. Assume $\Phi:\rS(\fkh)\rightarrow\mathfrak{U}(\fkh)$ is a $\fkh$-equivariant quantization of the Poisson algebra $\rS(\fkh)$ (see \eqref{Sym}). 
Analogously to the case of symbols, a $\fkh$-equivariant graded star product $\star_\Phi$ can be obtained on $\rS(\fkh)$, as the pullback of the product on $\mathfrak{U}(\fkh)\otimes\bbC[[\hbar]]$ by the map $\Phi_\hbar=(\Phi\otimes\Id)\circ\Im$, where $\Im:\rS(\fkh)\otimes\bbC[[\hbar]]\rightarrow \rS(\fkh)\otimes\bbC[[\hbar]]$ is the linear map defined by $(\bi\hbar)^k\Id$ on $\rS_k(\fkh)$. Denoting by $\tau$ and~$\g$ the anti-automorphisms of $\mathfrak{U}(\fkh)$ and $\rS(\fkh)$ defined by $-\Id$ on $\fkh$, the symmetry of the star product on $\rS(\fkh)$ is equivalent to $\Phi_\hbar(\bar{\cdot})=\tau\circ\Phi_\hbar(\cdot)$, or simply $\Phi\circ\g=\tau\circ\Phi$.

The regular functions on the coadjoint orbits of $H$ are Poisson algebras $\rS(\fkh)/I$,  for various ideals $I$. To build graded $\fkh$-equivariant star-products on them boils down to find $\fkh$-equivariant quantization maps $\rS(\fkh)/I\rightarrow\mathfrak{U}(\fkh)/J$ with $\Gr J\cong I$, see e.g.\ \cite{ABr02}. A method is to find a $\fkh$-equivariant quantization $\Phi$ of $\rS(\fkh)$, such that $\Phi(I)=J$. This is not trivial, $\Phi(I)$ is not an ideal of $\mathfrak{U}(\fkh)$  in general \cite{FLl01}. On minimal nilpotent coadjoint orbits, for $\fkh\neq \mathfrak{sl}(n)$ a simple Lie algebra, there exists a unique $\fkh$-equivariant quantization and a unique graded $\fkh$-equivariant star-product \cite{ABC94,ABr02}. In that case, $J$ is the Joseph ideal.

Here, we build a family of graded $\fkg$-equivariant star-products, out of a family of $\fkg$-equivariant quantization of $\rS(\fkg)$, on the coadjoint orbits $\cO_{a^{(\pm)}}$, $a\in\bbR$, and $\cO_{00}$, as given in Theorem \ref{thmCoad}. According to Theorem \ref{ThmK}, the algebra of regular functions on $\cO_{0^{\pm}}$ is the algebra $\cK\simeq \rS(\fkg)/\cI_0$ of generalized conformal Killing tensors.
By Diagram \eqref{PBW}, the $\fkg$-equivariant quantization $\cQ^{\l,\l}$ on $\cK$ lifts to a quantization on $\rS(\fkg)$ and induces a star-deformation of $\cK$. This extends to the algebra of regular functions $\rS(\fkg)/\cI_a$ on $\cO_{a^{(\pm)}}$ (see Proposition \ref{PolyO}) via the following Lemma.
\begin{lem}\label{IsoOrbit}
Let $a\in\bbR$. There exists a $\fkg$-equivariant linear map $\phi_a=\Id+N_a$ on $\rS(\fkg)$, such that $N_a$ lowers the degree and $\phi_a(\cI_a)=\cI_0$. Thus, we get $\rS(\fkg)/\cI_a\simeq\cK$ as $\fkg$-modules. 
\end{lem}
\begin{proof}
We know that $\rS(\fkg)\simeq I\oplus\cK$ and $I=(C)+\left({\Yboxdim{5pt}\yng(1,1,1,1)}\right)$. Resorting to the semi-simplicity of $\fkg$ and the filtration of 
$\cI_a=(C-a)+\left({\Yboxdim{5pt}\yng(1,1,1,1)}\right)$, we get that $\rS(\fkg)\simeq \cI_a+\rS(\fkg)/\cI_a$ and $(C-a)$ admits a $\fkg$-stable  complement in $\cI_a$. The map $\phi_a$ defined by $\frac{C}{C-a}\Id$ on $(C-a)$ and by the identity on a $\fkg$-stable complementary space satisfies the required properties.
\end{proof}
\begin{thm}
There exists a family of $\fkg$-equivariant quantizations $(\Phi_a^\l)_{a,\l\in\bbR}$ of $\rS(\fkg)$ such that: (i) it lifts $(\cQ^{\l,\l})_{\l\in\bbR}$ to $\rS(\fkg)$ for $a=0$, (ii) it induces a family of symmetric $\fkg$-invariant star products on the coadjoint orbits $\mathcal{O}_{a^{(\pm)}}$ for $a\in\bbR$, (iii) if $a=0$ and $\l=\frac{n-2}{2n}$, it induces the unique graded $\fkg$-equivariant star-product on the minimal coadjoint orbit $\mathcal{O}_{00}$.
\end{thm}
\begin{proof}
The Proposition \ref{PropIJl} ensures the existence of a $\fkg$-equivariant quantization $\Phi_\l$ of $\rS(\fkg)$ lifting $\cQ^{\l,\l}$ for every $\l\in\bbR$. The lift property is equivalent to $\Phi^\l(I)=J^\l$. We define then the family of $\fkg$-equivariant quantizations $\Phi^\l_a=\Phi^\l\circ\phi_a$, where $\phi_a$ is introduced in Lemma~\ref{IsoOrbit}. It can be chosen such that $\phi_0=\Id$, so (i) is trivially satisfied. 
The Lemma \ref{IsoOrbit} ensures that~$\Phi^\l_a(\cI_a)$ is an ideal and a $\fkg$-module, hence the $\fkg$-invariant star product $\star_{\Phi_a^\l}$ on $\rS(\fkg)$, induced by $\Phi_a^\l$, descends on the quotient $\rS(\fkg)/\cI_a$. We recall that  $\star_{\Phi_a^\l}$ is symmetric if $\Phi_a^\l$ satisfies $\tau\circ\Phi_a^\l=\Phi_a^\l\circ\g$. Redefining $\Phi_a^\l$ by $\half(\Phi_a^\l+\tau\circ\Phi_a^\l\circ\g)$, this is trivially the case, and the quantization $\Phi_0^\l$ is still a lift of $\cQ^{\l,\l}$ by uniqueness of the latter. This proves (ii). The last point follows then from Corollary \ref{PropPBW}, Proposition \ref{PolyO} and the uniqueness result in  \cite{ABC94,ABr02}. 
\end{proof} 
\begin{rmk}
For two distinct coadjoint orbits, the star products obtained above do not coincide in general. This is reminiscent to the work of Fioresi and Lled\'o \cite{FLl01}, dealing with star products tangential to semi-simple coadjoint orbits of semi-simple Lie groups.   
\end{rmk}
\begin{rmk}
Via the conformally equivariant quantization $\cQ^{\l,\l}$, the star-product on the minimal nilpotent coadjoint orbit of $\rO(p+1,q+1)$ is represented by the  algebra of differential operators preserving the kernel of the conformal Laplacian. The latter space is nothing else than the minimal unitary representation of  $\rO(p+1,q+1)$ if $p+q\geq 4$ is even \cite{BZi91}.
\end{rmk}
 \subsection*{Acknowledgements}
 It is a pleasure to acknowledge Christian Duval and Valentin Ovsienko for their constant interest in this work, and Josef \v{S}ilhan for invaluable discussions.

\bibliographystyle{plain}
\bibliography{Biblio}

\end{document}